\documentclass[english]{article}

\usepackage{mathrsfs}%
\usepackage[T1]{fontenc}
\usepackage{marvosym}
\usepackage{babel}
\usepackage{amsmath,amsfonts,amsthm}%
\usepackage{amsmath}%
\usepackage{amsfonts}%
\usepackage{amssymb}%
\usepackage{graphicx}

\newcommand{\vep}{\varepsilon}
\newcommand{\tX}{\tilde{X}}

\newcommand{\be}{\begin{equation}}
\newcommand{\ee}{\end{equation}}

\newtheorem{thm}{Theorem}[section]

\newtheorem{lem}[thm]{Lemma}
\newtheorem{coro}[thm]{Corollary}
\newtheorem{rem}[thm]{Remark}
\newtheorem{prop}[thm]{Proposition}

\newtheorem{defi}[thm]{Definition}

\newcommand{\E}{\hat{\mathbb{E}}}
\newcommand{\R}{\mathbb{R}}

\newcommand{\X}{\mathbb{X}}
\newcommand{\BX}{\mathbf{X}}

\newcommand{\oc}{\mathcal {C}}
\newcommand{\FC}{\mathscr{C}}
\newcommand{\B}{\mathbb{B}}
\newcommand{\BB}{\mathbf{B}}

\newcommand{\hc}{\hat{c}}
\newcommand{\tBX}{\tilde{\mathbf{X}}}

\begin{document}
\title{ Wong-Zakai Approximation for SDEs Driven by $G-$Brownian Motion }

\author{Shige Peng\\[7pt]
\small\itshape{School of Mathematics, Shandong University,}\\
\small\itshape{Jinan, China, peng@sdu.edu.cn} \and
Huilin Zhang\footnote{the corresponding author}\\
\small\itshape{School of Mathematics, Fudan University,}\\
\small\itshape{Shanghai, China, huilinzhang2014@gmail.com$^*$}}
\date{} \maketitle
\makeatletter

\newskip\@footindent
\@footindent=0em
\renewcommand\footnoterule{\kern-3\p@ \hrule width 0.4\columnwidth \kern 2.6\p@}

\long\def\@makefntext#1{\@setpar{\@@par\@tempdima \hsize
\advance\@tempdima-\@footindent
\parshape \@ne \@footindent \@tempdima}\par
\noindent \hbox to \z@{\hss\@thefnmark\hspace{0.2em}}#1}
\renewcommand\thefootnote{\myfootnotestyle{\arabic{footnote}}}
\def\@makefnmark{\hbox{\@thefnmark}}
\makeatother

\ \

\noindent\textbf{Abstract}: In this paper, we build the Wong-Zakai approximation for Stratonovich type
SDE driven by $G$-Brownian motion and obtain the quasi-surely convergence rate under H\"{o}lder norm by rough path argument. As a corollary, we obtain the quasi-continuity of solutions of random RDEs driven by lifted martingales under a sequence of singular measures.
\\

\ \

\\


\date{}
\maketitle

\section{Introduction}
Nonlinear expectation theory ($G$-expectation theory) is introduced by Peng in \cite{P07a,P08a,P10}. It is widely used as a helpful tool for financial problems concerning
model uncertainty and ambiguous volatility \cite{EJ13, V14}. Also, it provides a
coexistence framework for a set of mutually singular martingale
measures. However, according to probability uncertainty of the $G-$expectation theory, it is hard to
simulate stochastic differential equations driven by $G-$Brownian motion ($G-$SDEs) directly in the classical probabilistic way. For example, how to numerically generate a random variable sharing the same ``distribution'' with $G-$Brownian motion is still
an interesting and important problem in both $G-$expectation theory and application in real market (see recent progress in \cite{JP16, PZ20}). In this paper, authors provide an alternative way of calculating $G-$SDEs pathwisely.\\

Wong and Zakai first build the approximation of Stratonovich SDEs by
a sequence of ODEs in \cite{WZ65a,WZ65b}. In their papers, they give sufficient conditions under which solutions of ODEs converge to SDEs of Stratonovich's kind. As well-known, these conditions work well in one-dimensional case. For the high dimensional case, the approximation sensitively depends on how well the driven signals of ODEs $W_n$ converges to $W$, see Stroock and Varadhan \cite{SV72} for the high dimensional case. Lyons establishes the rough path theory in his seminal work \cite{L98}, and uses Wong-Zakai theorem to get the equivalence between rough differential equation (RDE) and SDE. The numerical analysis of Taylor's expansion kind ( Euler or Milstein's) approximation for SDEs driven by fractional Brownian
Motion is studied in \cite{DNT12} pathwisely and \cite{BFRS16} in
the sense of expectation. Recently, Kelly and Melbourne \cite{KM16} provide an approximation to SDEs through rough path method by constructing smooth and c\`adl\`ag approximations with smooth flows to (lifted) Brownian motion. Hairer and Pardoux \cite{HP15} give a positive answer to the Wong-Zakai kind approximation to SPDEs by the fast-developing regularity structure method. All these results imply rough path is a helpful tool for studying Wong-Zakai kind approximation, and then further
obtaining numerical results for SDEs driven by $G$-Brownian motion.\\

$G$-Brownian motion is first lifted as a geometric rough path under
$p$-variation norm in \cite{GQY14}. In that paper, authors also
establish the Euler-Maruyama approximation for $G-$SDEs. Later, authors of this paper follow notations from Gubinelli \cite{G04,G10}, and lift the $G$-Brownian motion by Kolmogorov's theorem (rough paths version in $G$-framework) in \cite{PZ16}.
Also, $G$-Stratonovich integral is introduced there. Followed by that work, in this paper we build the Wong-Zakai approximation for Stratonovich type SDE driven by $G$-Brownian motion and obtain the quasi-surely convergence rate under H\"{o}lder norm by rough path argument. To apply Wong-Zakai's argument for this purpose, the first problem is to show solutions of ODEs driven by piecewisely linearized $G$-Brownian motion actually belong to the random variable space $L_G^p$ in $G$-framework. More precisely, one can pathwisely define the solution $y^{(n)}_t(\omega)$, to the following ODE driven by piecewisely linearized $G$-Brownian motion $B_t^{(n)}$,
\begin{equation}  
dy_t^{(n)}= f(y_t^{(n)})d B_t^{(n)} + g(y_t^{(n)})d\langle B
\rangle_t + h(y_t^{(n)}) dt,
\end{equation}
and show that it has nice integrability. However, it is not clear whether the solution $y^{(n)}_t(\omega)$ actually belongs to the space $L_G^p$ or not, because of the lack of quasi-continuity (see Definition \ref{cts-qs}) of $y^{(n)}_t(\omega)$. (Ironically, it is well-known that ODEs solutions are not continuous with respect to the driven signal under the uniform topology, and rough path theory is built on the stronger $p-$variation or $1/p-$H\"older norm topology.) The property of quasi-continuity is a big difference between $G-$expectation theory and the second order backward stochastic differential equation theory, and only for elements in $L_G^p$ one can apply classical arguments in the sublinear framework (e.g. the uniform conditional expectation, martingale inequality, etc), which is vital for the proof of the approximation. To apply rough path argument to obtain the convergence rate, e.g. the uniform continuity theorem mentioned in \cite{FH14, LQ02}, we need to explicitly calculate the distance between the piecewisely linearized $G$-Brownian motion $B_t^{(n)}$ and $B$ under rough path topology. One should note that the canonical process $B_t$ now is no longer Gaussian but a (aggregated) martingale under a set of singular measures, so classical tools such as Wiener Chaos theory do not work anymore. Concerning these problems, firstly, we show that solutions of ODEs driven by piecewisely linearized $G$-Brownian motion actually belong to the random variable space in
$G$-framework by an approximation argument. The method is also applicable to $G-$ODEs driven by other piecewisely linearized $G-$martingales. Secondly, instead of classical methods from Wiener Chaos expansion
we do a direct calculation based on martingale properties under nonlinear framework and a sharp application of the algebraic property of rough paths. As a consequence, firstly, we establish the equivalence
between an aggregated RDE driven by lifted canonical process under a set of singular martingale measures and $G$-Stratonovich kind SDE, which is first established in \cite{L98} in the classical case and implied in $G-$framework in \cite{GQY14}. Secondly, according to the representation of random variable space $L_G^p$, i.e. Theorem \ref{Grep}, we obtained the quasi-continuity of solutions of such RDEs.\\

The paper is organized as the following. In Section 2, we recall
some basic notations in $G$-expectation theory and rough path theory.
Then in Section 3, we give the main results of this paper. Firstly,
we prove the Wong-Zakai theorem in $G$-framework. Secondly, we estimate
the convergence rate by the continuity theorem of the It\^{o}-Lyons
mapping.

\bigskip

{\bf Acknowledgements:} 
H.Z. is supported by the Youth Program of National Natural Science Foundation No. 11901104 and Postdoctoral Science Foundation of China.


\section{Notations and preliminaries about rough paths and $G$-expectation}
In this part, we review some definitions and conclusions about
$G$-expectation and rough path theory. See excellent lecture notes as \cite{FH14,FV10,LQ02,P07a,P10} for details.\\


\subsection{The rough path theory}

A (level-$2$) rough path on some interval $[0,T]$ with values in $\R^d$ includes
a continuous path $X:[0,T] \rightarrow \R^d,$ and a ``second order'' mapping $\mathbb{X}:[0,T]^2 \rightarrow
\R^d\otimes \R^d,$ which satisfies the algebraic condition,
\begin{equation}\label{chen}
\X_{s,t}-\X_{s,u}-\X_{u,t}= X_{s,u} \otimes X_{u,t},
\end{equation}
and the analytic condition
\be\label{norm-r}
\|X\|_\alpha :=\sup_{0\leq s <t \leq T}
\frac{|X_{s,t}|}{|t-s|^\alpha} < \infty,
\ \ \ 
\|\X\|_{2\alpha} := \sup_{0\leq s \neq t \leq T}
\frac{|\X_{s,t}|}{|t-s|^{2\alpha}}<\infty,
\ee
where $X_{s,t}:=X_t- X_s.$\\

 In the sequel, we always suppose $\alpha \in (\frac13,\frac12)$ for the need of well-established level-$2$ rough path theory. Denote $\oc^{\alpha}(\R^d)$ the space of $\alpha-$H\"older continuous paths and $\FC^{\alpha}(\R^d)$ the space of level-$2$ rough paths. For $\BX=(X,\X),\  \tBX=(\tX, \tilde{\X}) \in \FC^\alpha, $ we equip it with a homogeneous semi-norm 
$$
\|\BX \|_{\FC^{\alpha}}:= \| X\|_{\alpha}+(\| \X \|_{2\alpha}
)^{\frac12}.
$$
and an inhomogeneous metric,
$$
\varrho_\alpha(\BX,\tBX):=\|X-\tX\|_{\alpha} + \|\X-\tilde{\X}\|_{2\alpha}.
$$

A $\alpha-$H\"older continuous path $Y \in \oc^{\alpha}(\R^m) $ is said to be controlled by
$X \in \oc^{\alpha}(\R^d)  ,$ if there exists $Y' \in
\oc^\alpha([0,T], {L}(\R^d,\R^m)),$ such that the remainder
term
$$
R_{s,t}^Y:=Y_{s,t}-Y'_s X_{s,t},
$$
satisfies $\|R^Y\|_{2\alpha} < \infty,$ where $\| \cdot \|_{2\alpha}$ is defined as in \eqref{norm-r}.
Denote the collection of controlled rough paths by $\mathcal
{D}^{2\alpha}_X([0,T],\R^m).$ For $(Y,Y') \in \mathcal
{D}^{2\alpha}_X([0,T],\R^m),$ we define its semi-norm by
$\|Y,Y'\|_{X,2\alpha}:=\|Y'\|_\alpha + \|R^Y\|_{2\alpha}.$ For example, given any $F \in \oc_b^2(\R^d,\R^m),$ the set of
bounded functions from $\R^d$ to $\R^m$ with bounded derivatives up
to order 2, one can easily check that $(Y,Y'):=(F(X),DF(X)) \in
\mathcal {D}^{2\alpha}_X([0,T],\R^m).$

\subsection{The $G$-expectation theory}

Let $\Omega= \oc^0(\R^+,\R^d)$ the space of $\R^d$ valued
continuous paths $(\omega)_{t\geq0}$ vanishing at the origin. Denote
the coordinate process by $B_t$ and $u^{\varphi(\cdot)}(t,x)$ the
unique viscosity solution to the following $G$-heat equation with initial condition $\varphi.$
\begin{equation}\label{Gheat}
\partial_t u-G(D_x^2u)=0, \ u(0,x)=\varphi(x),
\end{equation}
where 
\be
G(A)=\frac12 \sup_{\gamma \in \Gamma} tr(A\gamma \gamma'), \text{for any}\ A
\in \mathbb{S}_d,
\ee
and $\Gamma $ is a bounded, convex and closed subset of $\R^{d\times d}.$ Define $L_{ip}(\Omega_T):=\{
\varphi(B_{t_1\wedge T},...,B_{t_k\wedge T}): k\in \mathbb{N},
t_1,...t_k \in [0,\infty), \varphi\in \oc_{b.Lip}(\R^{k \times d})
\}$ for any $T>0$ and $L_{ip}(\Omega):= \bigcup_{n=1}^\infty
L_{ip}(\Omega_n) .$ We define a mapping $\E$ from $L_{ip}(\Omega)$
to $\R$ by defining $\E[\varphi(B_t)]:=u^{\varphi(\cdot)}(t,0)$ and recursively solving the $G$-heat equation for general elements:
\begin{equation}
\E[\varphi(B_{t_1}, ...,B_{t_n}-B_{t_{n-1}})]:=
\E[\varphi^{t_n-t_{n-1}}(B_{t_1},...,B_{t_{n-1}}-B_{t_{n-2}})],
\end{equation}
where $\varphi^{t_n-t_{n-1}}(x_1,...,x_{n-1}):=u^{\varphi(x_1,...,x_{n-1},\cdot)}(t_n-t_{n-1},0).$ One can check that $\E[\cdot]$ is well defined and it is a sublinear expectation on $L_{ip}(\Omega).$ For each $p \geq 1,$ $L_G^p(\Omega_T)$ denotes the completion of the
linear space $L_{ip}(\Omega_T),$ under norm $\| \cdot
\|_{L_G^p}:=\{\E[|\cdot|^p]\}^{\frac{1}{p}}.$  
Denote the Wiener measure by $P^0.$ Then we have the following
representation of $\E$ according to \cite{DHP11}.

\begin{thm}\label{Grep}\quad
For any time sequence
$0=t_0 < t_1...<t_k ,$ one has
\begin{eqnarray*}
\E[\varphi(B_{t_0,t_1},...,B_{t_{k-1},t_k})] &=& \sup_{a \in
\mathcal {A}^{\Gamma} } E_{P^0}[\varphi(\int_0^{t_1}a_s
dB_s,...,\int_{t_{k-1}}^{t_{k}}a_s dB_s )]  \\
&=& \sup_{P^{a} \in \mathscr{P}^{\Gamma} }
E_{P^a}[\varphi(B_{t_0,t_1},...,B_{t_{k-1},t_k})] ,
\end{eqnarray*}

where $\mathcal {A}^{\Gamma}$ is the set of progressively measurable
processes with values in $\Gamma$ and $\mathscr{P}^\Gamma$ is the
set of laws of $\int_0^. a_s dB_s$ with $a \in \mathcal
{A}^{\Gamma}$ under Wiener measure. Furthermore,
$\mathscr{P}^{\Gamma}$ is tight.

\end{thm}

%
%

According to this representation, one could extend $\E$ from $L_G^p$ to any
Borel measurable random variable by defining
\be\label{E-sup}
\|\cdot\|_{\mathbb{L}^p}:=\sup_{P^{a} \in
\mathscr{P}^{\Gamma} } E_{P^a}^{\frac{1}{p}}[|\cdot|^p].
\ee

Next, we introduce the notation of {\it capacity} corresponding to the $G$-expectation
and give the description of $L_G^p.$ Denote $\mathscr{B}(\Omega_T)$ the Borel $\sigma$ algebra on $\Omega_T$. Define
\be
\hc(A):= \sup_{P \in \mathscr{P}^{\Gamma} }P(A), \ \ \text{for} \ A \in
\mathscr{B}(\Omega_T).
\ee
A property is said to hold ``quasi-surely''(q.s.) with respect to
$\hc$, if it holds true outside a $\hc$-polar set (Borel set with
capacity 0), and is denoted by $\hc-q.s..$\\
A process $Y$ on $[0,T]$ is said to be a quasi-surely modification
of another process $X,$ if for any $t \in [0,T]$
$$
Y_t=X_t, \ \ \ \hc-q.s..
$$
According to representation \eqref{Grep}, one probably would ask whether a measurable r.v. $\xi \in L_G^p$ if and only if 
$\| \xi \|_{\mathbb{L}^p}< \infty.$ However, the answer is no and for $\xi \in L_G^p,$ it has a uniform regularity property with respect to singular measures $\mathscr{P}^{\Gamma},$ known as {\it quasi-continuity}.

\begin{defi}\label{cts-qs}
Equip the space $\Omega_T$ with the uniform topology. A mapping $\xi$
on $\Omega_T$ with values in $\R $ is said to be quasi-continuous if
for any $\varepsilon >0,$ there exists an open set $O,$ with
$\hc(O)< \vep$ such that $\xi$ is continuous in $O^c.$
\end{defi}

\begin{defi}
One says that $\xi:\Omega_T \rightarrow \R$ has a quasi-continuous
version if there exists a quasi-continuous function $\eta,$ such that
$\xi=\eta,  \ \hc-q.s..$

\end{defi}

\begin{thm}\label{repG}
One has the following representation for $L_G^1,$
$$
L_G^1(\Omega_T)=\{ \xi \in \mathscr{B}(\Omega_T) : \xi \text{ has a
quasi-continuous version, } \lim_{n\rightarrow \infty} \|
|\xi |1_{\{| \xi |>n\} } \|_{\mathbb{L}^1}=0 \}.
$$
\end{thm}

%
%
%


Denote $M_G^{p,0}(0,T)$ the collection of processes with form
$$
\eta_t(\omega)=\sum_{i=0}^{N-1} \xi_i(\omega)1_{[t_i,t_{i+1})}(t),
$$
for a partition $\{0=t_0 <...<t_N=T\}$ and $\xi_i \in
L_{ip}(\Omega_{t_i}),i=0,...,N-1.$ Then denote by $M_G^{p}(0,T)$ the
completion of $M_G^{p,0}(0,T)$ under norm
$\|\cdot\|_{M_G^p}:=\{\E\int_0^T |\eta_s|^p ds\}^{\frac{1}{p}}.$
One can define stochastic integrals for elements in $M_G^p$ (see \cite{P07a},\cite{P08a},\cite{P10}). 
In the following, we denote $\bar{\sigma}=[\E|B_1|^2]^\frac12.$ For positive integers $k, n, m$ and a Euclidean space $H$, we write $C_b^k(\R^n, H)$ to denote functionals from $\R^n$ to $H$ with bounded derivatives up to $k-$the order, and $L(  \R^m) $ to denote linear functions on $\R^m.$

\section{Main Result}

\subsection{Wong-Zakai approximation in $G$-framework }

In this part, we consider piecewisely linearized approximation to $G-$SDEs. Suppose $B$ is the $d-$dimensional $G-$Brownian motion on $[0,1].$ $\{t^{(n)}_j\}_{j=0}^n$ is the partition with mesh size $1/n$
and $B_t^{(n)}$ is the piecewise linearization of $G$-Brownian
motion according to $\{t^{(n)}_j\}_{j=0}^n.$ Consider the
following ODEs with initial condition $y_0 \in \R^m,$
\begin{equation} \label{ode}
dY_t^{(n)}= f(Y_t^{(n)})d B_t^{(n)} + g(Y_t^{(n)})d\langle B
\rangle_t + h(Y_t^{(n)}) dt.
\end{equation}

It is clear that for any $n,$ $Y^{(n)}$ can be defined pathwisely. The following lemma proves that $Y^{(n)}_t(\omega)$ is indeed quasi-continuous as a function on $\Omega_T.$

%
%
%
%
%
%

\begin{lem}\label{included-in-LG}
Assume $f\in C_b^1(\R, L(\R^d)), \ g\in C_b^1(\R, L(\R^{d\times d})), \ h\in C_b^1(\R, \R) $. For any fixed $n,$ one has
$Y_t^{(n)} \in L_G^2(\Omega_{t_{j+1}^{(n)}}),$ for any
$t\in[t_j^{(n)},t_{j+1}^{(n)}].$

\end{lem}

\begin{proof}
By induction, one only needs to show for any $n\geq1,$ $j=0,...,n-1,$
$Y^{(n)}_{t^{(n)}_{j+1}} \in L_G^1(\Omega_{t_{j+1}^{(n)}}).$ In the
following proof, we omit $(n)$ in $Y^{(n)}$ and $t^{(n)}_{j}$ for
simplicity, i.e. suppose $Y_t$ solves the following $G-$stochastic ODE pathwisely,
\begin{equation}\label{ode1}
Y_t=Y_{t_j}+\frac{B_{t_j,t_{j+1}}}{t_{j+1}-t_j} \int_{t_j}^t
f(Y_s)ds + \int_{t_j}^t g(Y_s)d\langle B\rangle_s + \int_{t_j}^t
h(Y_s)ds,\ \ t\in[t_j,t_{j+1}],
\end{equation}
and $Y_{t_j}\in L_G^2(\Omega_{t_j}).$ Consider the discretization of \eqref{ode1},
\begin{eqnarray}\label{dode}
y_t^{(m)}&=&y^{(m)}_{k-1} + \frac{B_{t_j,t_{j+1}}}{t_{j+1}-t_j}
f(y_{k-1}^{(m)})(t-\tau_{k-1}^{(m)}) + g(y_{k-1}^{(m)})\langle
B\rangle_{\tau_{k-1}^{(m)},t} \nonumber \\
         &+& h(y_{k-1}^{(m)})
(t-\tau_{k-1}^{(m)}),\ \ t \in [\tau_{k-1}^{(m)},\tau_k^{(m)}],
\end{eqnarray}
where $\{\tau_k^{(m)}\}_{k=0}^m$ is the partition of $[t_j,t_{j+1}]$
with mesh-size $\frac{1}{m}$, and
$y_k^{(m)}:=y^{(m)}_{\tau_k^{(m)}},$ $y^{(m)}_{t_j}:=Y_{t_j}.$\\

It holds that
\begin{eqnarray*}
&&Y_{t_{j+1}}-y_{t_{j+1}}^{(m)}  =  \sum_{i=0}^{m-1}
[\frac{B_{t_j,t_{j+1}}}{t_{j+1}-t_j}
\int_{\tau_i^{(m)}}^{\tau_{i+1}^{(m)}} (f(Y_s)-f(y_i^{(m)}))ds\\
 &+&
\int_{\tau_i^{(m)}}^{\tau_{i+1}^{(m)}} (g(Y_s)-g(y_i^{(m)}))d\langle
B \rangle_s + \int_{\tau_i^{(m)}}^{\tau_{i+1}^{(m)}}
(h(Y_s)-h(y_i^{(m)}))ds]\\
&=& \sum_{i=0}^{m-1} \{ \frac{B_{t_j,t_{j+1}}}{t_{j+1}-t_j}
\int_{\tau_i^{(m)}}^{\tau_{i+1}^{(m)}} (f(Y_s)-f(y_s^{(m)}))ds +
\int_{\tau_i^{(m)}}^{\tau_{i+1}^{(m)}} (g(Y_s)-g(y_s^{(m)}))d\langle
B \rangle_s \\
&+& \int_{\tau_i^{(m)}}^{\tau_{i+1}^{(m)}} (h(Y_s)-h(y_s^{(m)}))ds +
\frac{B_{t_j,t_{j+1}}}{t_{j+1}-t_j}
\int_{\tau_i^{(m)}}^{\tau_{i+1}^{(m)}}
(f(y_s^{(m)})-f(y_i^{(m)}))ds \\
&+& \int_{\tau_i^{(m)}}^{\tau_{i+1}^{(m)}}
(g(y_s^{(m)})-g(y_i^{(m)}))d\langle B \rangle_s +
\int_{\tau_i^{(m)}}^{\tau_{i+1}^{(m)}} (h(y_s^{(m)})-h(y_i^{(m)}))ds
\}.
\end{eqnarray*}

By the uniform bound $|\langle B \rangle_t| \leq \bar{\sigma}^2 t,$ where $\bar{\sigma}= [\E|B_1|^2]^{\frac{1}{2}},$ and a direct calculation, one has
\begin{equation}\label{2}
\int_{\tau_i^{(m)}}^{\tau_{i+1}^{(m)}} |y_s^{(m)}-y_i^{(m)}| ds \leq
C
(1+|Y_{t_j}|)e^{C(1+\frac{|B_{t_j,t_{j+1}}|}{t_{j+1}-t_j})(t_{j+1}-t_j)}(\tau_{i+1}^{(m)}-\tau_i^{(m)})^2.
\end{equation}
From here on $C$ is a generic constant. According to \eqref{2} and Lipschitzness for $f,g,h,$ one obtains 
\begin{eqnarray*}
|Y_{t_{j+1}}-y_{t_{j+1}}^{(m)}| &\leq &C
(1+\frac{|B_{t_j,t_{j+1}}|}{t_{j+1}-t_j}) \int_{t_j}^{t_{j+1}} |Y_s
-y_s^{(m)}| ds \\
&+& C (1+|Y_{t_j}|)e^{C(1+\frac{|B_{t_j,t_{j+1}}|}{t_{j+1}-t_j})(t_{j+1}-t_j)}
\frac{1}{m} (t_{j+1}-t_j).
\end{eqnarray*}

Notice that $t_{j+1}$ can be replaced by any $t\in[t_j,t_{j+1}].$ By Gronwall's inequality, one has the following inequality,
$$
|Y_{t_{j+1}}-y_{t_{j+1}}^{(m)}| \leq \frac{C}{m}
(1+|Y_{t_j}|) e^{C((t_{j+1}-t_j)+
|B_{t_j,t_{j+1}}|)}.
$$
By taking expectation one obtains
$
\E|Y_{t_{j+1}}-y_{t_{j+1}}^{(m)}|^2 \leq \frac{C}{m^2},
$
which implies our result since $y_t^{(m)} \in L_G^2(\Omega_{t_{j+1}}),$ for any $ t \in
[t_j,t_{j+1}].$
\end{proof}

Thanks to Lemma \ref{included-in-LG}, the proof of the following theorem is an adaptation of the classical case in $G-$framework.

\begin{thm}\label{wzthm}
Suppose $f\in C_b^2(\R, L(\R^d)), \ g\in C_b^1(\R, L(\R^{d\times d})), \ h\in C_b^1(\R, \R) $, and $Y_t^{(n)}$ solves ODEs \eqref{ode} $\hc-$quasi surely. $X_t$ solves the following $G-$SDE of Stratonovich's kind,
\begin{equation}\label{gssde}
X_t=x_0 + \int_0^t f( X_s) \circ dB_s + \int_0^t g( X_s)d\langle
B\rangle_s + \int_0^t h( X_s) ds.
\end{equation}
Then for any $t\in [0,1],$ $Y_t^{(n)}$
converges to $X_t$ in $L_G^2$-norm sense. Furthermore, for any $t
\in [0,1],$ one has the following inequality,
\begin{equation}
\E[(Y_t^{(n)}-X_t)^2] \leq K \frac{1}{\sqrt{n}},
\end{equation}
where $K$ depends on $f,g,h$ and $\bar{\sigma}.$

\end{thm}

\begin{proof}

Without loss of generality, we suppose $g,h=0$ for simplicity.
In the following, $K$ denotes a generic constant and may be different from line
to line. Consider the Maruyama's approximation to $G$-Stratonovich SDE
\eqref{gssde},
\begin{equation}\label{maru}
dX_t^{(n)}=X_{j-1}^{(n)} + f(X_{j-1}^{(n)}) B_{t^{(n)}_{j-1},t} +
\frac12 Df(X_{j-1}^{(n)})f(X_{j-1}^{(n)})\langle B
\rangle_{t_{j-1}^{(n)},t}, \ \ t\in [t_{j-1}^{(n)},t_{j}^{(n)}],
\end{equation}
where $X_{j-1}^{(n)}=X^{(n)}_{t^{(n)}_{j-1}},$ $j=1...n.$ By Maruyama's approximation in $G$-framework, i.e. Theorem 7 in Part 3
of \cite{GQY14}, one have
$
\E[\sup_{t\in[0,T]}|X_t-X_t^{(n)}|^2] \leq K(\frac{1}{n})
$
so one only needs to show
\begin{equation}\label{wz1}
\E(Y_t^{(n)}- X_t^{(n)})^2 \leq K \frac{1}{\sqrt{n}}, \ \  \forall t
\in [0,1].
\end{equation}
Note that for any $t \in [t_{j}^{(n)}, t_{j+1}^{(n)}),$ one has estimates
\be
\E(Y_t^{(n)}-Y_j^{(n)})^2 \leq \frac{K}{n},\ \ \E(X_t^{(n)}-X_j^{(n)})^2 \leq \frac{K}{n}.
\ee
It suffices to prove \eqref{wz1} for $t=t_{j}^{(n)}.$ Firstly, one
has the following identity,
\begin{eqnarray}
Y_j^{(n)} - X_j^{(n)} &=& (Y_{j-1}^{(n)}-X_{j-1}^{(n)})+
(f(Y_{j-1}^{(n)}) - f(X_{j-1}^{(n)}))B_{t_{j-1}^{(n)}, t_j^{(n)}} \label{wze1}\\
&+&
\int_{t_{j-1}^{(n)}}^{t_j^{(n)}} (f(Y_s^{(n)})-f(Y_{j-1}^{(n)}))
ds \frac{B_{t_{j-1}^{(n)}, t_j^{(n)}}}{t_j^{(n)}- t_{j-1}^{(n)}} \label{wze3} \\
&-& \frac12 Df(X_{j-1}^{(n)})f(X_{j-1}^{(n)})\langle B
\rangle_{t_{j-1}^{(n)},t_j^{(n)}}. \label{wze4}
\end{eqnarray}
By Taylor's expansion, for \eqref{wze3}, one has
\begin{eqnarray}
&&
\int_{t_{j-1}^{(n)}}^{t_j^{(n)}} (f(Y_s^{(n)})-f(Y_{j-1}^{(n)}))
ds \frac{B_{t_{j-1}^{(n)}, t_j^{(n)}}}{t_j^{(n)}- t_{j-1}^{(n)}} \nonumber \\
&=&
\int_{t_{j-1}^{(n)}}^{t_j^{(n)}} f'(Y^{(n)}_{\tau_s}) (Y_s^{(n)}
-Y_{j-1}^{(n)} )ds \frac{B_{t_{j-1}^{(n)}, t_j^{(n)}}}{t_j^{(n)}- t_{j-1}^{(n)}} \nonumber \\
&=&  \int_{t_{j-1}^{(n)}}^{t_j^{(n)}}
Df(Y^{(n)}_{\tau_s}) \int^s_{t_{j-1}^{(n)}} f(Y_r^{(n)}) dr ds \Big(\frac{B_{t_{j-1}^{(n)}, t_j^{(n)}}}{t_j^{(n)}-
t_{j-1}^{(n)}}\Big)^{\otimes 2}
\label{wze33},
\end{eqnarray}
where $Y^{(n)}_{\tau_s}= Y^{(n)}_{j-1} + \theta (Y^{(n)}_{s} -
Y^{(n)}_{j-1}), $ $\theta \in (0,1).$
By subtracting \eqref{wze4} from \eqref{wze33} and inserting terms,
one obtains that
\begin{eqnarray*}
&&Y_j^{(n)} - X_j^{(n)}= (Y_{j-1}^{(n)}-X_{j-1}^{(n)})
+(f(Y_{j-1}^{(n)}) - f(X_{j-1}^{(n)}))B_{t_{j-1}^{(n)}, t_j^{(n)}}\\
&+& \frac12
[Df(Y^{(n)}_{j-1})f(Y^{(n)}_{j-1})-Df(X^{(n)}_{j-1})f(X^{(n)}_{j-1})]\langle B \rangle_{t_{j-1}^{(n)},t_j^{(n)}}\\
&+&\frac12
Df(Y^{(n)}_{j-1})f(Y^{(n)}_{j-1})(B_{t_{j-1}^{(n)}, t_j^{(n)}}^{\otimes2}-
\langle B \rangle_{t_{j-1}^{(n)},t_j^{(n)}})\\
&+&
\int_{t_{j-1}^{(n)}}^{t_j^{(n)}}(s-t_{j-1}^{(n)})[Df(Y^{(n)}_{\tau_s})f(Y^{(n)}_{\tau_s})-
Df(Y^{(n)}_{j-1})f(Y^{(n)}_{j-1})]ds \Big(\frac{B_{t_{j-1}^{(n)}, t_j^{(n)}}}{t_j^{(n)}-
t_{j-1}^{(n)}}\Big)^{\otimes2}\\
&+&
\int_{t_{j-1}^{(n)}}^{t_j^{(n)}}\int_{t_{j-1}^{(n)}}^{
s}Df(Y^{(n)}_{\tau_s})[f(Y^{(n)}_{r})-f(Y^{(n)}_{\tau_s})]drds\Big(\frac{B_{t_{j-1}^{(n)}, t_j^{(n)}}}{t_j^{(n)}-
t_{j-1}^{(n)}}\Big)^{\otimes2}.
\end{eqnarray*}
Denote the above six terms as $\vep^{(n,j)}_{l},$ $l=1...6.$
Firstly, by $f  \in \oc_b^2,$ it is clear that
\be
\E|\vep^{(n,j)}_{l}|^2  \leq  K (\frac{1}{n})^3, \ l=5,6, \ \ \E|\vep^{(n,j)}_{l}|^2  \leq K (\frac{1}{n})^2, \ \ l=4.
\ee
Secondly, by Lipschitzness of $f $ one has
$$
\E|\vep^{(n,j)}_{1}|^2+\E|\vep^{(n,j)}_{2}|^2+\E|\vep^{(n,j)}_{3}|^2+2
\E(\vep^{(n,j)}_{1}\vep^{(n,j)}_{3}) \leq (1+\frac{K}{n})
\E|Y_{j-1}^{(n)} - X_{j-1}^{(n)}|^2.
$$
By Lemma \ref{included-in-LG}, $Y^{(n)}_{k-1} \in
L_G^2(\Omega_{t^{(n)}_{k-1}}),$ which is independent from
$B_{t_{j-1}^{(n)}, t_j^{(n)}}$ and $\langle B
\rangle_{t_{j-1}^{(n)}, t_j^{(n)}}, $ so one gets
$
\E(\vep^{(n,j)}_{1} \vep^{(n,j)}_{l})=0, \ l=2,4.
$
Also, note that
\begin{eqnarray*}
\E| B_{t_{j-1}^{(n)}, t_j^{(n)}}[(B_{t_{j-1}^{(n)},
t_j^{(n)}})^{\otimes2}-\langle B \rangle_{t_{j-1}^{(n)},
t_j^{(n)}}]|&=&\E|[B_{t_{j-1}^{(n)}, t_j^{(n)}}
\int_{t_{j-1}^{(n)}}^{t_{j}^{(n)}} B_{t_{j-1}^{(n)},r}dB_r]| \leq  K (\frac{1}{n})^{\frac32},\\
 \E|\langle B \rangle_{t_{j-1}^{(n)}, t_j^{(n)}} [(B_{t_{j-1}^{(n)},
t_j^{(n)}})^{\otimes2}-\langle B \rangle_{t_{j-1}^{(n)}, t_j^{(n)}}]| &\leq& K
(\frac{1}{n})^{2},
\end{eqnarray*}
and one obtains
$$
|\E[\vep^{(n,j)}_{4}\vep^{(n,j)}_{l}]| \leq K (\frac{1}{n})^{\frac32},
\ \ l=2,3.
$$
As for other intersection terms concerning
$\vep^{(n,j)}_{5},\vep^{(n,j)}_{6},$ one applies H\"{o}lder's
inequality since $\E|\vep^{(n,j)}_{1}|^2 $ is bounded
and obtains
$$
\E[\vep^{(n,j)}_{k}\vep^{(n,j)}_{l}] \leq K (\frac{1}{n})^{\frac32},
\ \ l=1...6, k=5,6.
$$
 Finally, one gets
\begin{eqnarray*}
\E|Y_j^{(n)} - X_j^{(n)}|^2 &\leq& (1+\frac{K}{n}) \E|Y_{j-1}^{(n)}
- X_{j-1}^{(n)}|^2 + K (\frac{1}{n})^{\frac32} \label{wzl} \\
&\leq & \sum_{i=0}^{j-1} (1+\frac{K}{n})^{i} K
(\frac{1}{n})^{\frac32} \leq \frac{K}{\sqrt{n}}.
\end{eqnarray*}

\end{proof}

\subsection{The Convergence Rate under Uniform Norm for Wong-Zakai Approximation  }

In this part, we will calculate the quasi-surely convergence rate
for the Wong-Zakai approximation by rough path theory. To this end, according to the rough path argument, one needs to calculate $\varrho(\BB,\BB^{(n)}).$ 
For $X(\omega): [0,T] \rightarrow \R^d$ and $\X(\omega): [0,T]^2
\rightarrow \R^{d\times d}$ with $X_t \in L_G^q(\Omega_T),
\X_{s,t} \in L_G^{\frac{q}{2}}(\Omega_T),\forall s,t \in [0,T].$ If $\BX:=(X, \X)$
satisfies relation $(\ref{chen})$ quasi-surely, and furthermore, 
\begin{equation}\label{bound-lp}
\|X_{s,t}\|_{L_G^q} \leq C|t-s|^{\beta}, \;
\|\X_{s,t}\|_{L_G^{\frac{q}{2}}} \leq C |t-s|^{2\beta},
\end{equation}
with some constant C. Then according to Theorem 3.1 in \cite{PZ16}, $\BX \in \FC^{\alpha}$ quasi-surely for any $\alpha \in
(\frac13,\beta - \frac{1}{q}).$ Furthermore, we have the following Kolmogorov criterion under rough path distance in our nonlinear expectation framework, the proof of which follows similarly as in Theorem 3.3 of \cite{FH14}.


\begin{thm}\label{Kolc}

Suppose $\BX=(X, \X), \ \tilde{\BX}=(\tX, \tilde{\X})$ satisfy \eqref{chen} and \eqref{bound-lp}. Let $\Delta X= \tilde{X}- X$ and $\Delta \X= \tilde{\X}- \X
$ and assume that for some $\varepsilon >0,$ one has bounds,
\be
\|\Delta X_{s,t}\|_{L_G^q} \leq  \varepsilon |t-s|^\beta,\ \  \|\Delta
\X_{s,t}\|_{L_G^{\frac{q}{2}}} \leq  \varepsilon |t-s|^{2\beta}.
\ee
Then for any $\alpha \in [0, \beta-
\frac{1}{q}),$ there exists a constant $M>0,$ depending on $C,\alpha,\beta ,q,$
such that
\be\label{delta-x}
\|\|\Delta X \|_{\alpha}\|_{\mathbb{L}^q} \leq M \varepsilon,\ \ 
\|\|\Delta \X \|_{2\alpha}\|_{\mathbb{L}^{\frac{q}{2}}} \leq M
\varepsilon.
\ee
In particular, if $\|\Delta \X_{s,t}\|_{L_G^{q}} \leq  \varepsilon |t-s|^{2\beta}$ and $\beta - \frac{1}{q} > \frac13,$ then for any
$\alpha \in (\frac13, \beta - \frac{1}{q}),$
$\| \varrho_\alpha (\BX, \tilde{\BX} )  \|_{\mathbb{L}^q} \leq M
\vep.$
\end{thm}

Let $  \B^{strat}_{s,t}=\int_s^t B_{s,r} \circ dB_r.$ Denote $ \BB^{strat}=(B,\B^{strat}), $ which can be lifted to a rough path quasi-surely by the argument above Theorem \ref{Kolc}. Let $\B^{(n)}_{s,t}:= \int_s^t B^{(n)}_{s,r} \otimes   dB^{(n)}_r$ and $\BB^{(n)}=(B^{(n)}, \B^{(n)}).$
In the following we give a direct calculation of the quasi-surely convergence rate
for $\BB^{(n)}$ to $\BB^{strat}$ under $\mathbb{L}^p$ metric, by sharply applying the properties of $G-$Brownian motion, instead of the Gaussianity as in the classical case.

\begin{prop}\label{B converge}
Fix $\theta < \frac12 - \alpha $ and $q\geq 2,$ $\BB^{(n)}$ converges to $\BB^{strat}$ under
$\alpha-$H\"{o}lder rough norm in the $\mathbb{L}^q$ sense. Furthermore,
one has the following inequalities,
\begin{eqnarray}
\| \varrho_\alpha (\BB^{strat} , \BB^{(n)}) 
\|_{\mathbb{L}^q} \leq  K (\frac{1}{n})^\theta, \label{BB1} \\
\varrho_\alpha(\BB^{strat}, \BB^{(n)}) \leq M (\frac{1}{n})^\theta,
\ \ \hc-q.s., \label{BB2}
\end{eqnarray}
where $K$ depends on $\bar{\sigma},$ and $M$ depends on
$\bar{\sigma}$ and the path $\omega.$
\end{prop}

\begin{proof}
To show inequality \eqref{BB1}, according to Theorem \ref{Kolc},
by taking $\beta=\frac12- \theta,$ it suffices to show
\begin{equation}
\|  \Delta B^{(n)}_{s,t}  \|_{\mathbb{L}^q} \leq K (\frac{1}{n})^\theta (t-s)^{\frac12-\theta}, \ \ \  \|  \Delta \B^{(n)}_{s,t}  \|_{\mathbb{L}^q} \leq
K(\frac{1}{n})^\theta (t-s)^{1-2\theta}, \label{BB12}
\end{equation}
for any $q\geq 2,$ where $\Delta B^{(n)}=B-B^{(n)}$ and $\Delta \B=
\B-\B^{(n)}.$ For the first inequality, if $t_i^{(n)} \leq s <t \leq t_{i+1}^{(n)},$ for some $0\leq
i \leq n-1,$ one has
\begin{eqnarray*}
\Delta B^{(n)}_{s,t} = B_{s,t} -
\frac{t-s}{t_{i+1}^{(n)}-t_i^{(n)}}B_{t_i^{(n)},t_{i+1}^{(n)}}.
\end{eqnarray*}
It follows that
\begin{eqnarray*}
\E|\Delta B^{(n)}_{s,t}|^q &\leq& 2^q [\E|B_{s,t}|^q +
(\frac{t-s}{t_{i+1}^{(n)}-t_i^{(n)}})^q
\E|B_{t_i^{(n)},t_{i+1}^{(n)}}|^q ]\\
& \leq & C_q \bar{\sigma}^q [|t-s|^{\frac{q}{2} } + |t-s|^q
|t_{i+1}^{(n)}-t_i^{(n)}|^{-\frac{q}{2}}]\\
& \leq &    C_q
\bar{\sigma}^q |t-s|^{\frac{q}{2}-q\theta } (\frac{1}{n})^{q\theta},
\end{eqnarray*}
where $C_q$ depends only on $q.$ If $t_{i} \leq s < t_{i+1}<...<t_{k} <t \leq
t_{k+1},$ for some $0\leq i<k\leq n,$ (we omit $(n)$ in $t_j$). Notice that $
\Delta B_{s,t}^{(n)}= \Delta B_{s,t_{i+1}}^{(n)} + \Delta
B_{t_k,t}^{(n)},$
and one could obtain the desired result by the first case.\\
For the second inequality in \eqref{BB12}, suppose $t_{i_0}^{(n)} \leq \tau_0 := s<
\tau_1<...<\tau_{k-1}<t =: \tau_k \leq t_{j_0}^{(n)} ,$ for some
$0\leq i_0<j_0\leq n$ and $|t-s| \ge \frac1n$ (it is much easier if $|t-s| < \frac1n$, indeed, inequality \eqref{BB12} would follows from inequalities as  \eqref{32} and \eqref{33} with $\frac1n$ replaced by $t-s$), where $ \tau_i:=  t_{i_0+i}^{(n)},  \ 1 \le i <k $ for convenience. According to \eqref{chen}, one has the following identity,
\begin{equation}\label{diffchen}
\Delta \B_{s,t}^{ (n)}= \sum_{i=0}^{k-1} \Delta
\B_{\tau_i,\tau_{i+1}}^{ (n)} + \sum_{0\leq i<j \leq k-1}
(B_{\tau_i,\tau_{i+1}} \otimes B_{\tau_j,\tau_{j+1}}-
B^{(n)}_{\tau_i,\tau_{i+1}} \otimes B^{(n)}_{\tau_j,\tau_{j+1}} ).
\end{equation}

For elements $\Delta \B^{l,l,(n)}$ in the diagonal of matrix $\Delta \B^{(n)}_{s,t},$ one has
\be
\B^{l,l,(n)}_{\tau_i,\tau_{i+1}}=\int_{\tau_i}^{\tau_{i+1}}
B_{\tau_i,r}^{l,(n)} d B_r^{l,(n)}=\frac{(B^l_{\tau_i,\tau_{i+1}})^2}{2}, \
\ \ i=1,...,k-2,
\ee
Then the first sum in \eqref{diffchen} on the diagonal is
\begin{equation}
\sum_{i=0}^{k-1} \Delta \B_{\tau_i,\tau_{i+1}}^{(n)}=
\B^{strat}_{s,\tau_1}-\B^{(n)}_{s,\tau_1}+\B^{strat}_{\tau_{k-1},t}-\B^{(n)}_{\tau_{k-1},t}.
\end{equation}

According to Theorem \ref{Grep}, one has the
following inequalities,
\begin{eqnarray}\label{32}
\E|\B^{strat}_{s,\tau_1}|^q \leq C_q \bar{\sigma}^q |\tau_1-s|^q \leq C_q \bar{\sigma}^q (\frac{1}{n})^q,\\ \label{33}
\E|\B^{(n)}_{s,\tau_1}|^q \leq C_q \bar{\sigma}^{2q} |\tau_1-s|^q
\leq C_q \bar{\sigma}^{2q} (\frac{1}{n})^q,
\end{eqnarray}
and similar results for $\B^{strat}_{\tau_{k-1},t}$ and
$\B^{(n)}_{\tau_{k-1},t},$ which implies our desired estimate \eqref{BB12} for elements on the diagonal of $\Delta \B_{s,t}^{(n)}$. For elements $\Delta \B^{m,l,(n)}$ off the diagonal of matrix $\Delta \B^{(n)}_{s,t},$ notice that
$$
\B^{m,l,(n)}_{\tau_i,\tau_{i+1}} - \B^{m,l }_{\tau_i,\tau_{i+1}}= \frac12(\int_{\tau_i}^{\tau_{i+1}} B^l_{\tau_i,r} dB^m_r - \int_{\tau_i}^{\tau_{i+1}} B^m_{\tau_i,r} dB^l_r ).
$$
By applying B-D-G inequality in $G-$framework (see e.g. Remark 2.10 in \cite{HJPS2}), and H\"older inequality, it follows that
\begin{eqnarray}
&&\E|\B^{m,l,(n)}_{\tau_1,\tau_{k-1}}- \B^{m,l }_{\tau_1,\tau_{k-1}} |^q \leq C_q\left( \E|\sum_{i=1}^{k-2}\int_{\tau_i}^{\tau_{i+1}} B^l_{\tau_i,r} dB^m_r |^q + \E|\sum_{i=1}^{k-2}\int_{\tau_i}^{\tau_{i+1}} B^m_{\tau_i,r} dB^l_r|^q\right) \nonumber \\
&\leq& C_q  \left( \E| \sum_{i=1}^{k-2}\int_{\tau_i}^{\tau_{i+1}} (B^m_{\tau_i,r})^2 d\langle B^l\rangle_r |^{\frac q2} +\E| \sum_{i=1}^{k-2}\int_{\tau_i}^{\tau_{i+1}} (B^l_{\tau_i,r})^2 d\langle B^m\rangle_r |^{\frac q2} \right)\nonumber \\
&\le&  C_q \bar{\sigma}^q \Big(  \E (\sum_{i=1}^{k-2} \int_{\tau_i}^{\tau_{i+1}} |B^m_{\tau_i, r}|^q dr ) (\sum_{i=1}^{k-2} \int_{\tau_i}^{\tau_{i+1}}  dr)^{\frac{q}{2}-1} \nonumber\\
&&  +  \E(\sum_{i=1}^{k-2} \int_{\tau_i}^{\tau_{i+1}} |B^l_{\tau_i, r}|^q dr ) (\sum_{i=1}^{k-2} \int_{\tau_i}^{\tau_{i+1}}  dr)^{\frac{q}{2}-1} \Big) \nonumber \\
&\leq& C_q \bar{\sigma}^{2q} |t-s|^{\frac q 2 -1} \Big( \sum_{i=1}^{k-2} \int_{\tau_i}^{\tau_{i+1}} (r-\tau_i)^{\frac q 2} dr  \Big) \nonumber \\
&\leq& C_q \bar{\sigma}^{2q}   |t-s|^{\frac q 2} \frac{1}{n^\frac q 2}.\label{rema1.2}
\end{eqnarray}

As for the second sum in \eqref{diffchen}, note that
$B^{(n)}_{\tau_i}=B_{\tau_i}, \ i=1,...,k-1,$ and one obtains
\begin{eqnarray}
&&\sum_{0\leq i<j \leq k-1} (B_{\tau_i,\tau_{i+1}}
B_{\tau_j,\tau_{j+1}}- B^{(n)}_{\tau_i,\tau_{i+1}}
B^{(n)}_{\tau_j,\tau_{j+1}} ) \nonumber \\
&=& (B_{s,\tau_1}-B_{s,\tau_1}^{(n)}) B_{\tau_1,\tau_{k-1}} +
B_{\tau_1,\tau_{k-1}}(B_{\tau_{k-1},t}-B_{\tau_{k-1},t}^{(n)}) \nonumber \\
&&  +B_{s,\tau_1}B_{\tau_{k-1},t}  - B_{s,\tau_1}^{(n)}B_{\tau_{k-1},t}^{(n)} \nonumber \\
&=&(B_s^{(n)}-B_s)B_{\tau_1,\tau_{k-1}} +
B_{\tau_1,\tau_{k-1}}(B_t-B_t^{(n)}) \nonumber \\
&& +B_{s,\tau_1}B_{\tau_{k-1},t}-B_{s,\tau_1}^{(n)}B_{\tau_{k-1},t}^{(n)}.
\end{eqnarray}
By independence, one can check that
\begin{eqnarray}
\E|(B_s^{(n)}-B_s)B_{\tau_1,\tau_{k-1}}|^q &\leq&  C_q
\bar{\sigma}^{2q} (\frac{1}{n})^{\frac{q}{2}}|t-s|^{\frac{q}{2}}\\
\E|B_{s,\tau_1}B_{\tau_{k-1},t}|^q &\leq& \bar{\sigma}^{2q}
(\frac{1}{n})^q, \label{rema2}
\end{eqnarray}
and similar results for $B_{\tau_1,\tau_{k-1}}(B_t-B_t^{(n)})$ and
$B_{s,\tau_1}^{(n)}B_{\tau_{k-1},t}^{(n)}.$ \\
By \eqref{diffchen}-\eqref{rema2}, one obtains
\be
\| \Delta \B_{s,t}^{(n)}\|_{\mathbb{L}^q} \leq  C_q
(\bar{\sigma}^2+\bar{\sigma} )(\frac{1}{n})^{\frac12}
|t-s|^{\frac12} \le C_q (\bar{\sigma}^2+\bar{\sigma} ) (\frac{1}{n})^{\theta}
|t-s|^{1-2\theta},
\ee
where we applied $|t-s|\ge \frac1n$ in the last inequality.
This completes the proof of inequalities \eqref{BB12} and then \eqref{BB1}. 
Considering \eqref{BB2}, it follows by a classical Borel-Cantelli argument in G-framework. Indeed, for any $\theta < \frac12- \alpha,$ one may
choose $q>2$ and $\theta < \theta' <\frac12 - \alpha,$ such that
$q(\theta'-\theta)>2.$ It is clear that
\begin{eqnarray*}
&&\hc(\bigcap_{M=1} \bigcap_{n=1} \bigcup_{m\geq
n}\{\varrho_\alpha(\BB^{strat}, \BB^{(m)} )>M\frac{1}{m^\theta} \})\\
&\leq& \hc(  \bigcap_{n=1} \bigcup_{m\geq
n}\{\varrho_\alpha(\BB^{strat}, \BB^{(m)} )> \frac{1}{m^\theta} \})=\hc(\limsup_{m}\{\varrho_\alpha(\BB^{strat}, \BB^{(m)} )>
\frac{1}{m^\theta}\}).
\end{eqnarray*}
Note that
$$
\hc(\varrho_\alpha(\BB^{strat}, \BB^{(m)} )> \frac{1}{m^\theta})\leq K^q
\frac{(\frac{1}{m^{\theta'}})^q}{(\frac{1}{m^{\theta}})^q}= K^q
\frac{1}{m^{q(\theta'-\theta)}}.
$$
According to Borel-Cantelli lemma in G-framework (see Lemma 5 in
\cite{DHP11}), one obtains
$$
\hc(\bigcap_{M=1} \bigcap_{n=1} \bigcup_{m\geq
n}\{\varrho_\alpha(\BB^{strat}, \BB^{(m)} )>M\frac{1}{m^\theta}
\})=0,
$$
which implies the desired result.

\end{proof}

%
%

Here is the main result of this section.

\begin{thm}
Suppose $f\in C_b^3(\R, L(\R^d)), \ g\in C_b^3(\R, L(\R^{d\times d})), \ h\in C_b^3(\R, \R) $, and $Y^{(n)}$ defined as in
\eqref{ode}. $X$ solves the following G-Stratonovich SDE
$$
X_t=x_0 + \int_0^t f(X_s) \circ dB_s + \int_0^t g(X_s)d\langle
B\rangle_s + \int_0^t h(X_s) ds,
$$
and $Y$ solves the following RDE driven by
G-Stratonovich rough paths,
\begin{equation}
dY_t= f(Y_t) d\BB^{strat} + g(Y_t) d\langle B \rangle_t + h(Y_t)dt,
\end{equation}
with initial condition $x_0.$ Then for any $\theta <
\frac12-\alpha,$ one has the following inequality,
$$
\|Y-Y^{(n)}\|_\alpha \leq M(\omega) \frac{1}{n^\theta}, \ \ \
\hc-q.s..
$$
In particular, $X=Y,$  $\hc-q.s.,$ and
$$
\|X-Y^{(n)}\|_\alpha  \leq M(\omega) \frac{1}{n^\theta}, \ \ \
\hc-q.s..
$$

\end{thm}

\begin{proof}
According to Proposition \ref{B converge} and Theorem 7.5 in \cite{FH14}, one obtains
the first inequality. The particular case follows
from Theorem \ref{wzthm} and the fact that $X,Y$ are quasi-surely
continuous with respect to $t.$
\end{proof}

\begin{rem}
Similar as the classical case, in the above theorem, the best choice for this $\theta$ is taken when $\alpha=\frac13+, $ which means the convergence rate $\theta$ takes the value of any positive number smaller than $\frac16.$ However, the well-known best convergence rate $\frac12-$ can be achieved by applying rough paths with lower regularity, see \cite{FV10}.

\end{rem}

The following corollary implies the quasi-continuity of solutions of RDEs
driven by lifted $G$-Brownian motion with respect to uniform norm on
the canonical space.

\begin{coro}
$Y$ solves the RDE driven by
G-Stratonovich rough paths,
\begin{equation}
dY_t= f(Y_t) d\BB^{strat} + g(Y_t) d\langle B \rangle_t + h(Y_t)dt,
\end{equation}
with initial condition $x_0.$ Then for any
$t<T,$ $Y_t $ has a quasi-continuous version.
\end{coro}

\begin{proof}
Since $X_t= Y_t, \  \hc-q.s.$, and $X_t \in  L_G(\Omega_t),$
the result follows by the representation of
$L_G(\Omega_t),$ i.e. Theorem \ref{repG}.
\end{proof}


%



\renewcommand{\refname}{\large References}{\normalsize \ }

\end{document}